\documentclass{amsart}
\usepackage{amsmath,amssymb}

\newtheorem{theorem}{Theorem}[section]
\newtheorem{proposition}[theorem]{Proposition}

\theoremstyle{definition}

\theoremstyle{remark}

\numberwithin{equation}{section}

\begin{document}

\title[A Total Q-Curvature Upper Bound and its Isoperimetric
Deficit]{An Upper Bound of the Total Q-Curvature and Its
Isoperimetric Deficit for Higher-dimensional Conformal Euclidean
Metrics}


\thanks{J. Xiao was supported in part by Natural Science and
Engineering Research Council of Canada.}

\author{C. B. Ndiaye}
\address{Mathematisches Institut der Eberhard-Karls-Universit\"at T\"ubingen, Auf der Morgenstelle 10, D-72076 T\"ubingen, Germany.
}
\email{ndiaye@everest.mathematik.uni-tuebingen.de}

\author{J. Xiao}
\address{Department of Mathematics and Statistics, Memorial University of Newfoundland, St. John's, NL A1C 5S7, Canada}
\email{jxiao@mun.ca}


\date{}



\begin{abstract}
The aim of this paper is to give not only an explicit upper bound of
the total Q-curvature but also an induced isoperimetric deficit
formula for the complete conformal metrics on $\mathbb R^n$, $n\ge
3$ with scalar curvature being nonnegative near infinity and
Q-curvature being absolutely convergent.
\end{abstract}
\maketitle

\section{Introduction}

To begin with, let us agree to some basic conventions. We employ
the symbols $\Delta$ and $\nabla$ to denote the Laplace operator
$\sum_{k=1}^n\partial^2/\partial x_k^2$ and the gradient vector
$(\partial/\partial x_1,...,\partial/\partial x_n)$ over the
Euclidean space $\mathbb R^n$ (with $o$ as the origin), $n\ge 2$. For notational convenience
we use $X\lesssim Y$ as $X\le CY$ for a constant $C>0$. We always
assume that $u$ is a smooth real-valued function on $\mathbb R^n$,
written $u\in C^\infty(\mathbb R^n)$, and then it generates a
conformal metric $g=e^{2u}g_0$ which is indeed a conformal
deformation of the standard Euclidean metric $g_0=\sum_{k=1}^n
dx_k^2$. The volume and surface area elements of the metric $g$ are
given by
$$
dv_g=e^{nu}d\mathcal{H}^n\quad\hbox{and}\quad
ds_g=e^{(n-1)u}d\mathcal{H}^{n-1}
$$
where $\mathcal{H}^k$ stands for the $k$-dimensional Hausdorff
measure. So, the volume and surface area of the open ball $B_r(x)$
and its boundary $\partial B_r(x)$ with radius $r>0$ and center
$x\in\mathbb R^n$ have the following values:
$$
v_g\big(B_r(x)\big)=\int_{B_r(x)}e^{nu}\,d\mathcal{H}^n\quad\hbox{and}\quad
s_g\big(\partial B_r(x)\big)=\int_{\partial
B_r(x)}e^{(n-1)u}\,d\mathcal{H}^{n-1}.
$$
More importantly, this metric takes two kinds of nonlinear operators
as the simplest ways of describing the curvature of the Riemannian
manifold $(\mathbb R^n, g)$. One is the scalar curvature (or Ricci
scalar) function
$$
S_{g,n}(x)=-2(n-1)e^{-2u(x)}\Big(\Delta u(x)+\frac{n-2}{2}|\nabla
u(x)|^2\Big).
$$
The other is the so-called Q-curvature (or Paneitz curvature)
function which, according to Fefferman-Graham \cite{FeGr}, can be
determined by
$$
Q_{g,n}(x)=e^{-nu(x)}(-\Delta)^{n/2} u(x),
$$
whose even and odd cases $Q_{g,2m}$ and $Q_{g,2m-1}$ are regarded
respectively as differential and pseudo-differential operators.

Note that both the scalar curvature $2^{-1}S_{g,2}$ and the
Q-curvature $Q_{g,2}$ are equal to the classical Gaussian curvature
$K$ which completely characterizes the curvature of the
two-dimensional Riemannian manifold $(\mathbb R^2, g)$. So it is
quite natural to recall a fundamental inequality of Gauss-Bonnet
integral in the theory of complete surfaces of totally finite
Gaussian curvature. According to Cohn-Vossen \cite{CV}, we see that
if $g=e^{2u}g_0$ is complete and analytic then
$$
\int_{\mathbb R^2}|K|dv_g=\int_{\mathbb R^2}|(-\Delta
u)|d\mathcal{H}^2<\infty
$$
yields the following Gauss-Bonnet integral inequality
\begin{equation}\label{eq2}
\int_{\mathbb R^2}K\,dv_g=\int_{\mathbb R^2}(-\Delta
u)\,d\mathcal{H}^2\le 2\pi.
\end{equation}
Furthermore, according to Finn \cite{Fi} we get that if
$g=e^{2u}g_0$ is complete and normal then $\int_{\mathbb
R^2}|K|dv_g<\infty$ implies the so-called Finn's isoperimetric
deficit formula
\begin{equation}\label{eq3}
1-\frac{1}{2\pi}\int_{\mathbb
R^2}K\,dv_g=1-\frac{1}{2\pi}\int_{\mathbb R^2}(-\Delta
u)d\mathcal{H}^2=\lim_{r\to\infty}\frac{\big(s_g\big(\partial
B_r(o)\big)\big)^2}{4\pi v_g\big(B_r(o)\big)}.
\end{equation}

In their 2000 paper \cite{ChQiYa} (see also its follow-up \cite{ChQiYa1}), Chang, Qing and Yang extend the
above results (\ref{eq2}) and (\ref{eq3}) to $\mathbb R^4$ in terms of the scalar and Q
curvatures. More precisely, if $g=e^{2u}g_0$ is complete and its
scalar curvature $S_{g,4}$ is nonnegative near infinity, then
$$
\int_{\mathbb R^4}|Q_{g,4}|\,dv_g=\int_{\mathbb R^4}|(-\Delta)^2
u|\,d\mathcal{H}^4<\infty
$$
implies the following Chang-Qing-Yang's integral inequality of
Gauss-Bonnet-Chern type
\begin{equation}\label{eq5}
\int_{\mathbb R^4}Q_{g,4}\,dv_g=\int_{\mathbb R^4}(-\Delta)^2 u\,
d\mathcal{H}^4\le 8\pi^2
\end{equation}
and Chang-Qing-Yang's isoperimetric deficit formula
\begin{equation}\label{eq6}
1-\frac{1}{8\pi^2}\int_{\mathbb
R^4}Q_{g,4}\,dv_g=1-\frac{1}{8\pi^2}\int_{\mathbb R^4}(-\Delta)^2
u\,d\mathcal{H}^4=\lim_{r\to\infty}\frac{\big(s_{g}\big(\partial
B_r(o)\big)\big)^{4/3}}{4(2\pi^2)^{1/3}v_g\big(B_r(o)\big)}.
\end{equation}

In his 2005 paper \cite{Fa}, Fang generalizes (\ref{eq5}) but not
(\ref{eq6}) to the even-dimensional space $\mathbb R^{n}$.
Explicitly speaking, suppose $n\ge 4$ is even, if $g=e^{2u}g_0$ is
complete and its scalar curvature $S_{g,n}$ is nonnegative near
infinity, then
$$
\int_{\mathbb R^n}|Q_{g,n}|\,dv_g=\int_{\mathbb R^n}|(-\Delta)^{n/2}
u|\,d\mathcal{H}^n<\infty
$$
yields the following Fang's integral inequality of
Gauss-Bonnet-Chern type
\begin{equation}\label{eq8}
\int_{\mathbb R^n}Q_{g,n}\,dv_g=\int_{\mathbb R^n}(-\Delta)^{n/2}u\,
d\mathcal{H}^n\le 2^{n-1}(n/2-1)!\pi^{n/2}.
\end{equation}

In our current paper, we establish the odd-dimensional version of
(\ref{eq8}) (covering (\ref{eq5})) and the any-dimensional extension
of (\ref{eq6}). Actually, our main assertion is stated in such a way
that can cover all Riemannian manifolds $(\mathbb R^n,g)$ with $n\ge
3$.

\begin{theorem}\label{t1} Given an integer $n\ge 3$ and a function $u\in C^\infty(\mathbb R^n)$, let
$g=e^{2u}g_0$ be complete with
$$\liminf_{|x|\to\infty}S_{g,n}(x)\ge
0\quad\hbox{and}\quad \int_{\mathbb R^n}|Q_{g,n}|\,dv_g<\infty.
$$
Then
\begin{equation}\label{eq10}
\int_{\mathbb R^n}Q_{g,n}\,dv_g=\int_{\mathbb R^n}(-\Delta)^{n/2}u\,
d\mathcal{H}^n\le 2^{n-1}\Gamma(n/2)\pi^{n/2}
\end{equation}
and
\begin{equation}\label{eq11a}
1-\frac{\int_{\mathbb
R^n}Q_{g,n}\,dv_g}{2^{n-1}\Gamma(n/2)\pi^{n/2}}=1-\frac{\int_{\mathbb
R^n}(-\Delta)^{n/2}u\,d\mathcal{H}^n}{2^{n-1}\Gamma(n/2)\pi^{n/2}}=\lim_{r\to\infty}
\frac{\big(s_g\big(\partial
B_r(o)\big)\big)^{n/(n-1)}}{(n\omega_n^{1/n})^{n/(n-1)}v_g\big(B_r(o)\big)},
\end{equation}
hold, where $\omega_n=\frac{2\pi^{n/2}}{n\Gamma(n/2)}$ is the
Lebesgue measure of the unit ball $B_1(o)$.
\end{theorem}

Perhaps it is appropriate to make two remarks. The first one is that
(\ref{eq10}) gives an explicit upper bound (i.e.,
$2^{n-1}\Gamma(n/2)\pi^{n/2}$) of the total Q-curvature
$\int_{\mathbb R^n}Q_{g,n}\,dv_g$ in all dimensions $n\ge 3$, but
also can be used to confirm that \cite[Theorem 1.3]{BoHeSa} has an
odd-dimensional analogue -- that is -- for each odd number $n\ge 3$
there is a dimensional constant $L_n\ge 1$ such that every manifold
$(\mathbb R^n,g)$ is $L_n$-biLipschitz equivalent to the background
manifold $(\mathbb R^n,g_0)$ provided that $u\in C^\infty(\mathbb
R^n)$ satisfies:
$$
u(x)=\hbox{constant}+\big(2^{n-1}\Gamma(n/2)\pi^{n/2}\big)^{-1}\int_{\mathbb
R^n}\big(\log|y|/|x-y|\big)(-\Delta)^{n/2}u(y)\,d\mathcal{H}^n(y)
$$
and
$$
\big(2^{n-1}\Gamma(n/2)\pi^{n/2}\big)^{-1}\int_{\mathbb
R^n}|(-\Delta)^{n/2}u|\,d\mathcal{H}^n<n
2^{-(7+4n)}e^{-4n(n-1)}3^{-2n}<1.
$$
The second one is that as the geometrical isoperimetric deficit of
(\ref{eq10}) the generalized Gauss-Bonnet-Chern formula
(\ref{eq11a}) (which is unknown until now except $n=2,4$) has
suggested a geometric meaning of the so-called Q-curvature of any
$3\le n$-dimensional manifold $(\mathbb R^n,g)$ -- see also
\cite{Pet} for Chang's question on the geometric content of
Q-curvature as well as Yang's study plan on Q-curvature in
odd-dimensions (for which the Gauss-Bonnet-Chern formula does not
hold). To better understand this suggestion, a dedicated
investigation of the version of (\ref{eq11a}) over a complete $3\le
n$-manifold with only finitely many conformally flat simple ends
(extending two/four-dimensional results in \cite{Fi}/\cite{ChQiYa1}
and settling the equality case for the even-dimensional inequality
in \cite[Theorem 1.1]{Fa}) is worth being carried out -- after this
paper was completed, we were informed that X. Xu did this thing in
\cite{Xu2} plus proving Theorem \ref{t1} by other methods.

The proof of Theorem \ref{t1} is allocated to the forthcoming four
sections. Our argument techniques and methods (working for all
dimensions bigger than or equal to three; see also \cite{Nd} and
\cite{Xu}) come mainly from harmonic analysis based on the radially
symmetric integral estimates and calculations -- for example in
Proposition \ref{l1} (i)-(ii)-(iii) (for $\mathbb R^n$, $n\ge 3$) of
this paper there is no need to solve some induced ordinary
differential equations such as ones treated in
\cite[pp.526-531]{ChQiYa} (for $\mathbb R^4$) and \cite[p.478]{Fa}
(for $\mathbb R^{2m}$) -- this direct approach makes our work be
initially like no theirs. Here we want to acknowledge several
helpful communications with M. Bonk, H. Fang, R. Graham, J. Li and X. Xu. Moreover, we are grateful to A. Malchiodi and the referee for their nice suggestions on the paper.

\section{Proof of (\ref{eq10}) -- Special Case}

In this section we provide a proof of (\ref{eq10}) for the smooth
radially symmetric function.

\begin{proposition}\label{l1} Let $u\in C^\infty(\mathbb R^n)$ be radially
symmetric and satisfy the hypotheses of Theorem \ref{t1}. If
$$
v(x)=\frac{1}{2^{n-1}\Gamma(n/2)\pi^{n/2}}\int_{\mathbb
R^n}\Big(\log\frac{|y|}{|x-y|}\Big)(-\Delta)^{n/2}u(y)\,
d\mathcal{H}^n(y),
$$
then:

\item{\rm (i)}
$$
\sup_{0<|x|, |y|<\infty}\frac{1}{\mathcal{H}^{n-1}\big(\partial
B_{|x|}(o)\big)}\int_{\partial
B_{|x|}(o)}\frac{{\big||z|^2-|y|^2\big|}}{|z-y|^2}d\mathcal{H}^{n-1}(z)<\infty.
$$

\item{\rm(ii)} $v$ is also radially symmetric and enjoys
$$
\lim_{r\to 0}r \frac{dv(r)}{dr}=0\quad\hbox{and}\quad\lim_{r\to
\infty}r\frac{dv(r)}{dr}=-\frac{1}{2^{n-1}\Gamma(n/2)\pi^{n/2}}\int_{\mathbb
R^n}(-\Delta)^{n/2}u\, d\mathcal{H}^n.
$$

\item{\rm(iii)} $\limsup_{|x|\to\infty}|x||\nabla v(x)|<\infty$ and $\limsup_{|x|\to\infty}|x|^2|\Delta v(x)|<\infty$.

\item{\rm (iv)} In the sense of distribution,
$$
(-\Delta)^{n/2}(-\log|x-y|)=2^{n-1}\Gamma(n/2)\pi^{n/2}\delta_y(x),
$$
where $\delta_y(\cdot)$ is the Dirac measure at $y$.

\item{\rm(v)} There is a constant $c$ such that $u(x)=v(x)+c$ for all
$x\in\mathbb R^n$.

\item{\rm(vi)} (\ref{eq10}) holds.
\end{proposition}

\begin{proof} (i) Given $x,y\in\mathbb R^n$, for simplicity we not only assume
$$
I(|x|,|y|)=\frac{1}{\mathcal{H}^{n-1}\big(\partial
B_{|x|}(o)\big)}\int_{\partial
B_{|x|}(o)}\frac{\big||z|^2-|y|^2\big|}{|z-y|^2}d\mathcal{H}^{n-1}(z),
$$
but also split $\partial B_{|x|}(o)$ into two disjoint parts $P_1$ and $P_2$, where
$$
P_1=\big\{z\in \partial B_{|x|}(o):\ |x|^2+|y|^2\le|z-y|^2\big\}
$$
and
$$
P_2=\big\{z\in \partial B_{|x|}(o):\ |x|^2+|y|^2>|z-y|^2\big\}.
$$
Due to the structure of $P_2$, we further write $P_2$ as the
union of countable disjoint sets as follows:
$$
P_2=\bigcup_{k\ge 0}\Big\{z\in\partial B_{|x|}(o):\,\
2^{-k-1}\le\frac{|z-y|}{\sqrt{|x|^2+|y|^2}}<2^{-k}\Big\}.
$$
Based on the spherical coordinate system on $P_2$ and
the law of cosines for the triangle formed by vectors $z\in P_2$,
$y$ and $z-y$, we define
$$
\phi_k=\arccos\frac{(1-2^{-2k})(|x|^2+|y|^2)}{2|x||y|},\quad
k=0,1,2,...,\left[\frac{\log\frac{(|x|-|y|)^2}{|x|^2+|y|^2}}{\log\frac14}\right].
$$

After the above technical treatment, we now need to deal with two
cases $n=3$ and $n>3$ respectively.

 {\it Case 1: $n=3$.} Under this case, we put
$$
H(|x|,|y|)=\frac{1}{\mathcal{H}^{2}\big(\partial B_{|x|}(o)\big)}\int_{\partial
B_{|x|}(o)}|z-y|^{-1}\,d\mathcal{H}^{2}(z)
$$
and then prove
$$
\sup_{x,y\in\mathbb R^3}(|x|+|y|)H(|x|,|y|)<\infty
$$
through handling two subcases.

{\it Subcase 1: $|y|\le |x|$.} When $|y|\le|x|/2$, we obviously have that $|z-y|\ge|z|-|y|\ge|x|/2$ as $z\in\partial B_{|x|}(o)$ and so that
$$
(|x|+|y|)H(|x|,|y|)\lesssim|x|^{-1}\int_{\partial
B_{|x|}(o)}|z-y|^{-1}\,d\mathcal{H}^{2}(z)\lesssim 1.
$$
Suppose now $|x|/2<|y|\le |x|$. Then we use $\partial B_{|x|}(o)=P_1\cup P_2$ and $\phi_k$ to estimate

\begin{eqnarray*}
&&(|x|+|y|)H(|x|,|y|)\\
&&\lesssim |x|^{-1}\left(\int_{P_1}|z-y|^{-1}\,d\mathcal{H}^{2}(z)+\int_{P_2}|z-y|^{-1}\,d\mathcal{H}^{2}(z)\right)\\
&&\lesssim 1+|x|^{-1}\sum_{k\ge 0}\int_{\{z\in\partial
B_{|x|}(o):\ 2^{-k-1}\le\frac{|z-y|}{\sqrt{|x|^2+|y|^2}}<2^{-k}\}}|z-y|^{-1}\,d\mathcal{H}^{2}(z)\\
&&\lesssim 1+\sum_{k\ge 0}2^k\int_{\phi_{k+1}}^{\phi_k}\sin\phi\,d\phi\\
&&\lesssim 1+\sum_{k\ge 0}2^{-k}.
\end{eqnarray*}

{\it Subcase 2: $|y|>|x|$.} When $|y|\ge 2|x|$, we similarly have that $|z-y|\ge|y|-|z|\ge|y|/2$ as $z\in\partial B_{|x|}(o)$, and so that
$$
(|x|+|y|)H(|x|,|y|)\lesssim\frac{|y|}{|x|^2}\int_{\partial
B_{|x|}(o)}|z-y|^{-1}\,d\mathcal{H}^{2}(z)\lesssim 1.
$$
If $|x|<|y|<2|x|$, then

\begin{eqnarray*}
&&(|x|+|y|)H(|x|,|y|)\\
&&\lesssim |y|^{-1}\left(\int_{P_1}|z-y|^{-1}\,d\mathcal{H}^{2}(z)+\int_{P_2}|z-y|^{-1}\,d\mathcal{H}^{2}(z)\right)\\
&&\lesssim 1+|y|^{-1}\sum_{k\ge 0}\int_{\{z\in\partial B_{|x|}(o):\ 2^{-k-1}\le\frac{|z-y|}{\sqrt{|x|^2+|y|^2}}<2^{-k}\}}|z-y|^{-1}\,d\mathcal{H}^{2}(z)\\
&&\lesssim 1+\sum_{k\ge 0}2^k\int_{\phi_{k+1}}^{\phi_k}\sin\phi\,d\phi\\
&&\lesssim 1+\sum_{k\ge 0}2^{-k}.
\end{eqnarray*}

The previous consideration of two subcases, plus the inequality
$$
\big||z|^2-|y|^2\big|/|z-y|^2\le(|z|+|y|)/|z-y|,
$$
leads to
$$
\sup_{0<|x|,|y|<\infty}I(|x|,|y|)\lesssim\sup_{x,y\in\mathbb R^3}(|x|+|y|)H(|x|,|y|)<\infty.
$$

{\it Case 2: $n>3$.} Under this case, we set
$$
J(|x|,|y|)=\frac{1}{\mathcal{H}^{n-1}\big(\partial
B_{|x|}(o)\big)}\int_{\partial
B_{|x|}(o)}{|z-y|^{-2}}d\mathcal{H}^{n-1}(z)
$$
and are about to show
$$
\sup_{x,y\in\mathbb R^n}(|x|^2+|y|^2)J(|x|,|y|)<\infty
$$
via handling two more subcases.

{\it Subcase 1: $|y|\le |x|$.} When $|y|\le |x|/2$, we have that $z\in\partial B_{|x|}(o)$ implies $|z-y|\ge|z|-|y|\ge|x|/2$ and consequently,
$$
(|x|^2+|y|^2)J(|x|,|y|)\lesssim|x|^2J(|x|,|y|)\lesssim
|x|^{1-n}\int_{\partial
B_{|x|}(o)}\frac{|x|^2}{|z-y|^2}d\mathcal{H}^{n-1}(z)\lesssim 1.
$$
When $|x|/2<|y|\le |x|$, we continue using the spherical coordinate system to produce
\begin{eqnarray*}
&&(|x|^2+|y|^2)J(|x|,|y|)\\
&&\lesssim|x|^2J(|x|,|y|)\\
&&\lesssim|x|^{1-n}\left(\int_{P_1}\frac{|x|^2}{|z-y|^2}\,d\mathcal{H}^{n-1}(z)+\int_{P_2}\frac{|x|^2}{|z-y|^2}\,d\mathcal{H}^{n-1}(z)\right)\\
&&\lesssim 1+|x|^{1-n}\int_{P_2}\frac{|x|^2}{|z-y|^2}d\mathcal{H}^{n-1}(z)\\
&&\lesssim 1+|x|^{3-n}\sum_{k\ge 0}\int_{\{z\in\partial
B_{|x|}(o):\ 2^{-k-1}\le\frac{|z-y|}{\sqrt{|x|^2+|y|^2}}<2^{-k}\}}|z-y|^{-2}\,d\mathcal{H}^{n-1}(z)\\
&&\lesssim 1+|x|^{3-n}\sum_{k\ge 0}\int_{\phi_{k+1}}^{\phi_k}(|x|^2+|y|^2)^{-1}2^{2k}|x|^{n-1}\sin^{n-2}\phi\,d\phi\\
&&\lesssim 1+\sum_{k\ge 0}
2^{2k}\int_{\phi_{k+1}}^{\phi_k}\sin^{n-2}\phi
\,d\phi\\
&&\lesssim1+\sum_{k\ge 0} 2^{2k}\int_{\phi_{k+1}}^{\phi_k}(1-\cos^2\phi)^{(n-3)/2}\,d(-\cos\phi)\\
&&\lesssim1+\sum_{k\ge 0}
2^{2k}\int_{\frac{(1-2^{-2k})(|x|^2+|y|^2)}{2|x||y|}}^{\frac{(1-2^{-2(k+1)})(|x|^2+|y|^2)}{2|x||y|}}(1-t^2)^{(n-3)/2}dt\\
&&\lesssim1+\sum_{k\ge 0} 2^{2k}\left(1-\Big(\frac{(1-2^{-2k})(|x|^2+|y|^2)}{2|x||y|}\Big)^2\right)^{(n-3)/2}\Big(\frac{2^{-2k}(|x|^2+|y|^2)}{|x||y|}\Big)\\
&&\lesssim1+\sum_{k=0}^\infty 2^{-k(n-3)}.
\end{eqnarray*}

{\it Subcase 2: $|y|>|x|$.} When $|y|\ge 2|x|$, we clearly see that $z\in\partial B_{|x|}(o)$ implies $|z-y|\ge|y|-|z|\ge|y|/2$ and consequently,
$$
(|x|^2+|y|^2)J(|x|,|y|)\lesssim|y|^2J(|x|,|y|)\lesssim|x|^{1-n}\int_{\partial
B_{|x|}(o)}\frac{|y|^2}{|z-y|^2}d\mathcal{H}^{n-1}(z)\lesssim 1.
$$
When $|x|<|y|<2|x|$, we analogously derive

\begin{eqnarray*}
&&(|x|^2+|y|^2)J(|x|,|y|)\\
&&\lesssim|y|^2J(|x|,|y|)\\
&&\lesssim |x|^{1-n}\left(\int_{P_1}\frac{|y|^2}{|z-y|^2}d\mathcal{H}^{n-1}(z)+\int_{P_2}\frac{|y|^2}{|z-y|^2}d\mathcal{H}^{n-1}(z)\right)\\
&&\lesssim 1+|x|^{1-n}\sum_{k\ge 0}\int_{\{z\in\partial B_{|x|}(o):\ 2^{-(k+1)}\le\frac{|z-y|}{\sqrt{|x|^2+|y|^2}}<2^{-k}\}}\frac{|y|^2}{|z-y|^2}\,d\mathcal{H}^{n-1}(z)\\
&&\lesssim1+\sum_{k\ge 0}
\frac{2^{2k}|y|^2}{|x|^2+|y|^2}\int_{\phi_{k+1}}^{\phi_k}\sin^{n-2}\phi\,
d\phi\\
&&\lesssim1+\sum_{k\ge 0}2^{-k(n-3)}.
\end{eqnarray*}

Taking the foregoing inequalities for $J(|x|,|y|)$ into account, we get the desired finiteness:
$$
\sup_{0<|x|,|y|<\infty}I(|x|,|y|)\le\sup_{x,y\in\mathbb R^n}(|x|^2+|y|^2)J(|x|,|y|)<\infty.
$$

(ii) The radial symmetry of $v$ follows easily from the assumption
that $u$ is radially symmetric. Using $|x|=r>0$ we calculate

$$
\frac{d}{dr}\log\frac{|y|}{|x-y|}=\frac{-\frac{d}{dr}|x-y|^2}{2|x-y|^2}=\frac{|y|^2-|x|^2-|x-y|^2}{2|x||x-y|^2}
$$
and then employ the radial symmetry of $u$ to obtain
\begin{eqnarray*}
&&-{2^{n}\Gamma(n/2)\pi^{n/2}}\Big(r\frac{dv(r)}{dr}\Big)\\
&&=\int_{\mathbb
R^n}\Big(\frac{|x|^2-|y|^2+|x-y|^2}{|x-y|^2}\Big)(-\Delta)^{n/2}u(y)\,
d\mathcal{H}^n(y)\\
&&=\int_{\mathbb R^n}\left(1+\frac{\int_{\partial B_{|x|}(o)}
\frac{|z|^2-|y|^2}{|z-y|^2}\,d\mathcal{H}^{n-1}(z)}{\mathcal{H}^{n-1}\big(\partial
B_{|x|}(o)\big)}\right)(-\Delta)^{n/2}u(y)\, d\mathcal{H}^n(y).
\end{eqnarray*}
Because both (i) and
$$
\int_{\mathbb R^n}|(-\Delta)^{n/2}u|d\mathcal{H}^n=\int_{\mathbb
R^n}|Q_{g,n}|dv_g<\infty
$$
guarantee
\begin{eqnarray*}
&&\int_{\mathbb R^n}\left|\left(1+\frac{\int_{\partial B_{|x|}(o)}
\frac{|z|^2-|y|^2}{|z-y|^2}\,d\mathcal{H}^{n-1}(z)}{\mathcal{H}^{n-1}\big(\partial
B_{|x|}(o)\big)}\right)(-\Delta)^{n/2}u(y)\right|\,
d\mathcal{H}^n(y)\\
&&\lesssim\int_{\mathbb
R^n}\Big(1+\sup_{0<|x|,|y|<\infty}I(|x|,|y|)\Big)\big|(-\Delta)^{n/2}u(y)\big|\,
d\mathcal{H}^n(y)\\
&&\lesssim\int_{\mathbb
R^n}\big|(-\Delta)^{n/2}u\big|d\mathcal{H}^n<\infty,
\end{eqnarray*}
we apply the dominated convergence theorem to derive
\begin{eqnarray*}
&&\lim_{r\to
0}r\frac{dv(r)}{dr}\\
&&=-\frac{1}{2^{n}\Gamma(n/2)\pi^{n/2}}\int_{\mathbb
R^n}\lim_{|x|\to
0}\Big(\frac{|x|^2-|y|^2+|x-y|^2}{|x-y|^2}\Big)(-\Delta)^{n/2}u(y)\,
d\mathcal{H}^n(y)\\
&&=0
\end{eqnarray*}
and
\begin{eqnarray*}
&&\lim_{r\to
\infty}r\frac{dv(r)}{dr}\\
&&=-\frac{1}{2^{n}\Gamma(n/2)\pi^{n/2}}\int_{\mathbb
R^n}\lim_{|x|\to
\infty}\Big(\frac{|x|^2-|y|^2+|x-y|^2}{|x-y|^2}\Big)(-\Delta)^{n/2}u(y)\,
d\mathcal{H}^n(y)\\
&&=-\frac{1}{2^{n-1}\Gamma(n/2)\pi^{n/2}}\int_{\mathbb
R^n}(-\Delta)^{n/2}u\, d\mathcal{H}^n,
\end{eqnarray*}
as required.

(iii) The first finiteness follows from (ii) right away since
$\nabla$ can rewritten as $(d/dr, r^{-1}\nabla_\sigma)$ under the
spherical coordinate system where $\nabla_\sigma$ is the gradient
operator on the unit sphere $\partial B_1(o)$. To verify the second
finiteness, we observe (via an easy computation)
$$
\Delta v(x)=\frac{n-2}{2^{n-1}\Gamma(n/2)\pi^{n/2}}\int_{\mathbb
R^n}|x-y|^{-2}(-\Delta)^{n/2}u(y)\, d\mathcal{H}^n(y),
$$
and then handle two cases.

{\it Case 1: $n=3$.} From the hypotheses and the spherical coordinate
system it follows that
\begin{eqnarray*}
&&|x|^2|\Delta v(x)|\\
&&\lesssim\left(\int_{\{y\in\mathbb R^3:\
|y-x|\ge|x|/2\}}+\int_{\{y\in\mathbb R^3:\
|y-x|<|x|/2\}}\right)\frac{|x|^2}{|x-y|^2}|(-\Delta)^{3/2}u(y)|\,
d\mathcal{H}^3(y)\\
&&\lesssim\int_{\mathbb R^3}|(-\Delta)^{3/2}u|\,
d\mathcal{H}^3+\int_{\{y\in\mathbb R^3:\ |y|\ge
|x|/2\}}\frac{|x|^2}{|x-y|^2}|(-\Delta)^{3/2}u(y)|\,
d\mathcal{H}^3(y)\\
&&\lesssim\int_{\mathbb R^3}|(-\Delta)^{3/2}u|\,
d\mathcal{H}^3+\int_{\{y\in\mathbb R^3:\ 2|x|\ge|y|\ge|x|/2\}}\frac{|x|^2}{|x-y|^2}|(-\Delta)^{3/2}u(y)|\,d\mathcal{H}^3(y).\\
\end{eqnarray*}
Furthermore, via the spherical coordinate system we deduce

\begin{eqnarray*}
&&|x|^2\int_{\{y\in\mathbb R^3:\ |x|\ge|y|\ge|x|/2\}}|x-y|^{-2}|(-\Delta)^{3/2}u(y)|\,d\mathcal{H}^3(y)\\
&&\lesssim|x|^2\int_{|x|/2}^{|x|}|(-\Delta)^{3/2}u(t)|t^2\left(\int_0^{\pi/2}\frac{\sin\phi}{|x|^2-2|x|t\cos\phi+t^2}\,d\phi\right)dt\\
&&\lesssim|x|^2\int_{|x|/2}^{|x|}|(-\Delta)^{3/2}u(t)|t^2\left(\int_0^1\frac{ds}{|x|^2-2|x|ts+t^2}\right)dt\\
&&\lesssim\int_{\mathbb
R^3}|(-\Delta)^{3/2}u|d\mathcal{H}^3+\int_{|x|/2}^{|x|}|(-\Delta)^{3/2}u(t)|t^2\left(\log\frac{|x|^2+t^2}{(|x|-t)^2}\right)dt.
\end{eqnarray*}
Suppose now $U(t)=\int_0^t|(-\Delta)^{3/2}u(s)|s^2ds$ for $t>0$.
Integration by part and change of variables give
\begin{eqnarray*}
&&\int_{|x|/2}^{|x|}|(-\Delta)^{3/2}u(t)|t^2\left(\log\frac{|x|}{|x|-t}\right)\,dt\\
&&=-\int_0^{1/2}(\log s)\,d\big(U(|x|)-U(|x|-|x|s)\big)\\
&&=(\log
2)\big(U(|x|)-U(|x|/2)\big)+\int_0^{1/2}\big(U(|x|)-U(|x|-|x|s)\big)s^{-1}ds.
\end{eqnarray*}
Note that
$$
U(\infty)=\lim_{t\to\infty}U(t)\lesssim \int_{\mathbb R^3}|(-\Delta)^{3/2}u|\,d\mathcal{H}^3<\infty
$$
implies $\lim_{|x|\to\infty}\big(U(|x|)-U(|x|/2)\big)=0$. So the
limit-sup form of Fatou's lemma yields
\begin{eqnarray*}
&&0\le\limsup_{|x|\to\infty}\int_0^{1/2}\big(U(|x|)-U(|x|-|x|s)\big)s^{-1}ds\\
&&\le\int_0^{1/2}\limsup_{|x|\to\infty}\big(U(|x|)-U(|x|-|x|s)\big)s^{-1}ds=0.
\end{eqnarray*}
As a result, we get
$$
\limsup_{|x|\to\infty}|x|^2\int_{\{y\in\mathbb R^3:\
|x|\ge|y|\ge|x|/2\}}|x-y|^{-2}|(-\Delta)^{3/2}u(y)|\,d\mathcal{H}^3(y)=0.
$$
In a similar manner, we can also obtain
$$
\limsup_{|x|\to\infty}|x|^2\int_{\{y\in\mathbb R^3:\ |x|\le|y|\le
2|x|\}}|x-y|^{-2}|(-\Delta)^{3/2}u(y)|\,d\mathcal{H}^3(y)=0,
$$
thereby reaching
$$
\limsup_{|x|\to\infty}|x|^2|\Delta v(x)|\lesssim\int_{\mathbb R^3}|(-\Delta)^{3/2}u|\,
d\mathcal{H}^3<\infty.
$$

{\it Case 2: $n>3$.} Since $(-\Delta)^{n/2}u$ is radially symmetric, it
follows from the estimates on $J(|x|,|y|)$ that
\begin{eqnarray*}
&&|x|^2|\Delta v(x)|\\
&&\lesssim|x|^2\int_{\mathbb
R^n}\left(\frac{1}{\mathcal{H}^{n-1}\big(\partial
B_{|x|}(o)\big)}\int_{\partial
B_{|x|}(o)}\frac{d\mathcal{H}^{n-1}(z)}{|z-y|^2}\right)|(-\Delta)^{n/2}u(y)|d\mathcal{H}^n(y)\\
&&\lesssim\int_{\mathbb
R^n}|x|^2J(|x|,|y|)|(-\Delta)^{n/2}u(y)|d\mathcal{H}^n(y)\\
&&\lesssim\int_{\mathbb R^n}|(-\Delta)^{n/2}u|\,d\mathcal{H}^n,
\end{eqnarray*}
and so that
$$
\limsup_{|x|\to\infty}|x|^2|\Delta v(x)|\lesssim\int_{\mathbb R^n}|(-\Delta)^{n/2}u|\,d\mathcal{H}^n<\infty.
$$

(iv) In the even case this result may be found in \cite{BoHeSa}. A
proof of this case and odd one is provided below. Of course, it
suffices to verify the formula for $y=0$. Rewriting $\Delta$ in
terms of the spherical coordinate: $x=r\sigma$; $r>0$,
$\sigma\in\partial B_1(o)$, we have
$$
\Delta=\frac{d^2}{dr^2}+\Big(\frac{n-1}{r}\Big)\frac{d}{dr}+\frac{\Delta_\sigma}{r^2},
$$
where $\Delta_\sigma$ is the Laplacian on $\partial B_1(o)$. Since
$\log|x|$ is radially symmetric, if $n=2m$ is an even number, then a
simple calculation with the basic equation (see also \cite[p. 156,
(1)]{LieLo})
$$
(-\Delta)|x|^{2-n}=\frac{2(n-2)\pi^{n/2}}{\Gamma(n/2)}\delta_0(x)
$$
gives
\begin{eqnarray*}
&&(-\Delta)^{n/2}(-\log|x|)\\
&&=2\cdot 4\cdots
2(m-2)(2-n)(4-n)\cdots(2m-2-n)(-\Delta)|x|^{2-n}\\
&&=2^{n-1}\Gamma(n/2)\pi^{n/2}\delta_0(x).
\end{eqnarray*}
In the case that $n=2m-1$ is an odd number, a similar computation
yields
$$
(-\Delta)^{m-1}(-\log|x|)=2\cdot 4\cdots
2(m-2)(2-n)(4-n)\cdots(2m-2-n)|x|^{-2(m-1)}.
$$
According to \cite[p. 128, (2.10.1) \& (2.10.8)]{KiSrTr} and
\cite[p. 132, (3)]{LieLo}, we find
\begin{eqnarray*}
&&(-\Delta)^{-1/2}|x|^{-2(m-1)}\\
&&=\frac{\Gamma(n/2-1/2)}{2\pi^{(n+1)/2}}\int_{\mathbb R^n}|x-y|^{1-n}|y|^{2-2m}d\mathcal{H}^n(y)\\
&&=\Big(\frac{\sqrt{\pi}\Gamma(n/2-1)}{2\Gamma(n/2-1/2)}\Big)|x|^{2-n},
\end{eqnarray*}
whence obtaining (via the above-mentioned basic equation)
\begin{eqnarray*}
&&(-\Delta)^{n/2}(-\log|x|)\\
&&=(-\Delta)(-\Delta)^{-1/2}(-\Delta)^{m-1}(-\log|x|)\\
&&=2\cdot4\cdots(n-3)(2-n)(4-n)\cdots(-1)\Big(\frac{\sqrt{\pi}\Gamma(n/2-1)}{2\Gamma(n/2-1/2)}\Big)(-\Delta)|x|^{2-n}\\
&&=2^{n-1}\Gamma(n/2)\pi^{n/2}\delta_0(x).
\end{eqnarray*}

(v) From (iv) we see immediately that
$(-\Delta)^{n/2}v=(-\Delta)^{n/2}u$. To further get a constant $c$
such that $u=v+c$, we consider two situations.

{\it Situation 1: $n=2m-1$ is an odd integer.} Because
$(-\Delta)^{n/2}=(-\Delta)^{1/2}(-\Delta)^{m-1}$, $(-\Delta)^{n/2}v=(-\Delta)^{n/2}u$ yields $(-\Delta)^{1/2}(-\Delta)^{m-1}(v-u)=0$. Taking the Fourier transform of the fractional-order operator $(-\Delta)^{1/2}$ in the last equation, we find that the Fourier transform of $(-\Delta)^{m-1}(v-u)$ vanishes in $\mathbb R^n\setminus\{o\}$, thereby getting $(-\Delta)^{m-1}(v-u)=0$ via the inverse Fourier transform. Since $u-v$ is radially symmetric, we are
required to seek the radially symmetric solutions to
$(-\Delta)^{m-1}w=0$. Under the spherical coordinate system the last
equation becomes a linear ordinary differential equation (in the
radius $r=|x|$) of order $2(m-1)$. It is plain to check that
$2(m-1)$ functions $1,\log r, r^{\pm 2},...,r^{\pm 2(m-2)}$ satisfy
the equation but also are linearly independent. Thus there are
$2(m-1)$ constants $c_0, c_1, c_{\pm 2},...,c_{\pm 2(m-2)}$ such that
$$
v-u=c_0+c_1\log r+\sum_{k=1}^{m-2}(c_{2k} r^{2k}+c_{-2k}r^{-2k}).
$$
Thanks to the smoothness of $u$ and the first limit established in
(ii), we find $\lim_{r\to 0}rd(v-u)(r)/dr=0$ and consequently
$c_1=0$ as well as $c_{-2k}=0$ for $k=1,...,m-2$. On the other hand,
suppose $N$ is the largest integer amongst $\{1,...,m-2\}$ such that
$c_{2N}$ is nonzero. Then
$$
v-u=c_0+\sum_{k=1}^N c_{2k}r^{2k},
$$
and hence according to (iii) we have
$$
\limsup_{|x|\to\infty}\Big(|x|^2\Delta u(x)+(n/2-1)\big(|x||\nabla
u(x)|\big)^2\Big)=(n/2-1)c_{2N}^2\limsup_{r\to\infty}r^{4N}=\infty,
$$
But, nevertheless the hypothesis
$$
0\le\liminf_{|x|\to\infty}S_{g,n}(x)=-2(n-1)\limsup_{|x|\to\infty}e^{-2u}\Big(\Delta
u+(n/2-1)|\nabla u|^2\Big)
$$
amounts to
$$
\limsup_{|x|\to\infty}e^{-2u}\Big(\Delta u+(n/2-1)|\nabla
u|^2\Big)\le 0.
$$
With the above analysis, we reach a contradiction:
$$
\infty=\limsup_{|x|\to\infty}|x|^2\Big(\Delta u+(n/2-1)|\nabla
u|^2\Big)\le 0.
$$
Therefore $c_{2k}=0$ for all $k=1,...,m-2$. Consequently, $u=v-c_0$.

{\it Situation 2: $n=2m$ is an even integer.} Then
$$
(-\Delta)^{n/2}(v-u)=(-\Delta)^m(v-u)=0,
$$
and hence the previous argument for $n=2m-1$ can be employed to
deduce the result; see also \cite{Fa}.

(vi) Using (v) and the second limit in (ii) we obtain
$$
\lim_{r\to\infty}r\frac{du(r)}{dr}=\lim_{r\to\infty}r\frac{dv(r)}{dr}=-\frac{1}{2^{n-1}\Gamma(n/2)\pi^{n/2}}\int_{\mathbb
R^n}(-\Delta)^{n/2}u\, d\mathcal{H}^n,
$$
whence having
$$
\exp\big(u(r)\big)=\Big(\exp\big(u(1)\big)\Big)r^{-\frac{1}{2^{n-1}\Gamma(n/2)\pi^{n/2}}\int_{\mathbb
R^n}(-\Delta)^{n/2}u\, d\mathcal{H}^n +o(1)}\quad\hbox{as}\quad
r\to\infty.
$$
This last assertion, plus the hypothesis that $g=e^{2u}g_0$ is complete, ensures
$$
-\frac{1}{2^{n-1}\Gamma(n/2)\pi^{n/2}}\int_{\mathbb
R^n}(-\Delta)^{n/2}u\, d\mathcal{H}^n\ge -1,
$$
thereby implying (\ref{eq10}).

\end{proof}

\section{Proof of (\ref{eq10}) -- General Case}

In this section we prove (\ref{eq10}) through the radial symmetrization and
Proposition \ref{l1}. Although our argument ideas are similar
to ones explored in \cite{ChQiYa} and \cite{Fa}, for the paper's
completeness and the reader's convenience we feel that it is worth detailing the key steps of the proof.

\begin{proposition}\label{l2} Let $u\in C^\infty(\mathbb R^n)$ satisfy the hypotheses of Theorem \ref{t1}.
If
$$
\bar{u}(x)=\frac{1}{\mathcal{H}^{n-1}\big(\partial
B_{|x|}(o)\big)}\int_{\partial B_{|x|}(o)}u\,d\mathcal{H}^{n-1},
$$
then:

\item{\rm(i)} There is a constant $c$ such that
$$
u(x)=c+\frac{1}{2^{n-1}\Gamma(n/2)\pi^{n/2}}\int_{\mathbb
R^n}\Big(\log\frac{|y|}{|x-y|}\Big)(-\Delta)^{n/2}u(y)\,
d\mathcal{H}^n(y).
$$

\item{\rm(ii)} For any $p>0$ one has
$$
\lim_{|x|\to\infty}e^{-p\bar{u}(x)}\left(\frac{1}{\mathcal{H}^{n-1}\big(\partial
B_{|x|}(o)\big)}\int_{\partial B_{|x|}(o)}e^{pu}\,
d\mathcal{H}^{n-1}\right)=1.
$$

\item{\rm(iii)} $\bar{g}=e^{2\bar{u}}g_0$ is not only complete but also satisfies
$$
\liminf_{|x|\to\infty}S_{\bar{g},n}(x)\ge 0\quad\hbox{and}\quad\int_{\mathbb
R^n}|Q_{\bar{g},n}|dv_{\bar{g}}<\infty.
$$

\item{\rm(iv)} (\ref{eq10}) holds.
\end{proposition}

\begin{proof} (i) Continuing the use of $v$ defined in Proposition \ref{l1}, we get from Proposition \ref{l1} (iv) that
$(-\Delta)^{n/2}(u-v)=0$. To reach the desired result, we fix a
point $x_0\in\mathbb R^n$ and consider the radially symmetric versions
of $u$, $v$ and $u-v$ about $x_0$ as follows:

\[
\left\{\begin{array} {r@{\;\quad}l}

\bar{u}(x;x_0)=\frac{\int_{\partial B_{|x-x_0|}(x_0)}u\,d\mathcal{H}^{n-1}}{\mathcal{H}^{n-1}\big(\partial
B_{|x-x_0|}(x_0)\big)}, &\\
\bar{v}(x;x_0)=\frac{\int_{\partial
B_{|x-x_0|}(x_0)}v\,d\mathcal{H}^{n-1}}{\mathcal{H}^{n-1}\big(\partial
B_{|x-x_0|}(x_0)\big)}, &\\
\overline{u-v}(x;x_0)=\frac{\int_{\partial
B_{|x-x_0|}(x_0)}(u-v)\,d\mathcal{H}^{n-1}}{\mathcal{H}^{n-1}\big(\partial
B_{|x-x_0|}(x_0)\big)}.
\end{array}
\right.
\]
Owing to
\begin{eqnarray*}
v(x)\ =&\frac{1}{2^{n-1}\Gamma(n/2)\pi^{n/2}}\int_{\mathbb
R^n}\Big(\log\frac{|y-x_0|}{|x-y|}\Big)(-\Delta)^{n/2}u(y)\,
d\mathcal{H}^n(y)\\
\ +&\frac{1}{2^{n-1}\Gamma(n/2)\pi^{n/2}}\int_{\mathbb
R^n}\Big(\log\frac{|y|}{|y-x_0|}\Big)(-\Delta)^{n/2}u(y)\,
d\mathcal{H}^n(y),
\end{eqnarray*}
we see from the proof of the forthcoming (iii) that the
conformal metric $g_{x_0}=e^{2\bar{u}(\cdot;x_0)}g_0$ ensures
$$
\liminf_{|x|\to\infty}S_{g_{x_0},n}(x)\ge 0\quad\hbox{and}\quad\int_{\mathbb
R^n}|Q_{g_{x_0},n}|dv_{g_{x_0}}<\infty.
$$
So, by Proposition \ref{l1} (iv) and (v) we find that
$\overline{u-v}(x;x_0)$ equals a constant -- this especially derives
$\Delta (u-v)(x_0)=\Delta (\overline{u-v})(x_0;x_0)=0$. Since $x_0$
is arbitrarily chosen, one has $\Delta (u-v)=0$, i.e., $u-v$ is a
harmonic function on $\mathbb R^n$ and consequently, $\partial
(u-v)/\partial x_k$ (for each $k=1,...,n$) is harmonic. A combined
application of the mean-value-theorem, Cauchy-Schwarz's inequality
and the representation of $v$ yields
\begin{eqnarray*}
&&\Big|\frac{\partial(u-v)}{\partial x_k}(x_0)\Big|^2\\
&&=\left| \frac{1}{\mathcal{H}^{n-1}\big(\partial
B_{r}(x_0)\big)}\int_{\partial
B_{r}(x_0)}\frac{\partial(u-v)}{\partial
x_k}\,d\mathcal{H}^{n-1}\right|^2\\
&&\le\left|\frac{1}{\mathcal{H}^{n-1}\big(\partial
B_{r}(x_0)\big)}\int_{\partial B_{r}(x_0)}|\nabla(u-v)|\,d\mathcal{H}^{n-1}\right|^2\\
&&\lesssim\frac{1}{\mathcal{H}^{n-1}\big(\partial
B_{r}(x_0)\big)}\int_{\partial B_{r}(x_0)}(|\nabla u|^2+|\nabla
v|^2)\,d\mathcal{H}^{n-1}.
\end{eqnarray*}

Now, the representation of $v$, the Cauchy-Schwarz inequality, Fubini's theorem and the proof of Proposition \ref{l1} (iii) produce
\begin{eqnarray*}
&&\limsup_{r\to\infty}r^2\left(\frac{1}{\mathcal{H}^{n-1}\big(\partial
B_{r}(x_0)\big)}\int_{\partial B_{r}(x_0)}|\nabla v|^2\,d\mathcal{H}^{n-1}\right)\\
&&\lesssim\limsup_{r\to\infty}r^2\left(\frac{\int_{\partial B_r(x_0)}\Big(\int_{\mathbb R^n}\frac{|(-\Delta)^{n/2}u(y)|}{|z-y|^2}\,d\mathcal{H}^n(y)\Big)\,d\mathcal{H}^{n-1}(z)}{\mathcal{H}^{n-1}\big(\partial
B_{r}(x_0)\big)\Big(\int_{\mathbb
R^n}|(-\Delta)^{n/2}u|\,d\mathcal{H}^n\Big)^{-1}}\right)\\
&&<\infty.
\end{eqnarray*}
In the meantime, the formula of $S_{g,n}$, $\liminf_{|z|\to\infty} S_{g,n}(z)\ge 0$ and $\Delta u=\Delta v$ ensure that if $r\to\infty$ then
\begin{eqnarray*}
&&\frac{1}{\mathcal{H}^{n-1}\big(\partial
B_{r}(x_0)\big)}\int_{\partial B_{r}(x_0)}|\nabla u|^2\,d\mathcal{H}^{n-1}\\
&&=\Big(\frac{2}{2-n}\Big)\frac{1}{\mathcal{H}^{n-1}\big(\partial
B_{r}(x_0)\big)}\int_{\partial B_{r}(x_0)}\Big(\Delta u+\frac{e^{2u}S_{g,n}}{2(n-1)}\Big)\,d\mathcal{H}^{n-1}\\
&&\le\Big(\frac{2}{2-n}\Big)\frac{1}{\mathcal{H}^{n-1}\big(\partial
B_{r}(x_0)\big)}\int_{\partial B_{r}(x_0)}\Delta v\,d\mathcal{H}^{n-1}\\
&&=-\frac{(2^{n-2}\Gamma(n/2)\pi^{n/2})^{-1}}{\mathcal{H}^{n-1}\big(\partial
B_{r}(x_0)\big)}\int_{\partial B_{r}(x_0)}\left(\int_{\mathbb
R^n}\frac{(-\Delta)^{n/2}u(y)}{|z-y|^2}\,d\mathcal{H}^n(y)\right)\,d\mathcal{H}^{n-1}(z),
\end{eqnarray*}
and hence by Fubini's theorem and the proof of Proposition \ref{l1} (iii),
$$
\limsup_{r\to\infty}r^2\left(\frac{1}{\mathcal{H}^{n-1}\big(\partial
B_{r}(x_0)\big)}\int_{\partial B_{r}(x_0)}|\nabla
u|^2\,d\mathcal{H}^{n-1}\right)<\infty.
$$
Therefore,
$\limsup_{r\to\infty}r^2\Big|\frac{\partial(u-v)}{\partial
x_k}(x_0)\Big|<\infty$. This forces that $u-v$ is a constant.

(ii) The argument comes from a non-essential adaption of the proof
of \cite[Lemma 3.2]{ChQiYa}. According to the just-established (i),
we write $u=u_1+u_2$ where
$$
u_1(x)=c+\frac{1}{2^{n-1}\Gamma(n/2)\pi^{n/2}}\int_{B_{|x|/2}(o)}\Big(\log\frac{|y|}{|x-y|}\Big)(-\Delta)^{n/2}u(y)\,
d\mathcal{H}^n(y)
$$
and
$$
u_2(x)=\frac{1}{2^{n-1}\Gamma(n/2)\pi^{n/2}}\int_{\mathbb
R^n\setminus
B_{|x|/2}(o)}\Big(\log\frac{|y|}{|x-y|}\Big)(-\Delta)^{n/2}u(y)\,
d\mathcal{H}^n(y).
$$
If $u_1$ is further split into two pieces
$$
u_{11}(x)=\frac{1}{2^{n-1}\Gamma(n/2)\pi^{n/2}}\int_{B_{|x|/2}(o)}\Big(\log\frac{|y|}{|x|}\Big)(-\Delta)^{n/2}u(y)\,
d\mathcal{H}^n(y)
$$
and
$$
u_{12}(x)=\frac{1}{2^{n-1}\Gamma(n/2)\pi^{n/2}}\int_{B_{|x|/2}(o)}\Big(\log\frac{|x|}{|x-y|}\Big)(-\Delta)^{n/2}u(y)\,
d\mathcal{H}^n(y),
$$
then $u_1=u_{11}+u_{12}$, $u_{11}(x)=u_{11}(|x|)$, and
$\lim_{|x|\to\infty}u_{12}(x)=0$ -- this is because
$$
u_{12}(x)\lesssim\log(1-\epsilon)^{-1}+\int_{B_{|x|/2}(o)\setminus
B_{\epsilon|x|}(o)}\big|(-\Delta)^{n/2}u\big|\, d\mathcal{H}^n\to 0
$$
when $\epsilon\to 0$ is taken so that $\epsilon|x|\to\infty$ as
$|x|\to\infty$. As a result, we find
\begin{eqnarray*}
&&\frac{p}{\mathcal{H}^{n-1}\big(\partial
B_{|x|}(o)\big)}\int_{\partial
B_{|x|}(o)}(u-u_2)\, d\mathcal{H}^{n-1}\\
&&=\log\Big(\frac{1}{\mathcal{H}^{n-1}\big(\partial
B_{|x|}(o)\big)}\int_{\partial B_{|x|}(o)}\exp\big(p(u-u_2)\big)\,
d\mathcal{H}^{n-1}\Big)+o(1).
\end{eqnarray*}
On the one hand, we can make the following estimates for any $r\in
(0,\infty)$ and suitably small $\theta\in (0,1/2)$:
\begin{eqnarray*}
&&\left|\frac{1}{\mathcal{H}^{n-1}\big(\partial
B_{r}(o)\big)}\int_{\partial
B_{r}(o)}u_2\, d\mathcal{H}^{n-1}\right|\\
&&\lesssim\left|\int_{\mathbb R^n\setminus
B_{r/2}(o)}\left(r^{1-n}\int_{\partial
B_{r}(o)}\log\frac{|y|}{|x-y|}\,
d\mathcal{H}^{n-1}(x)\right)(-\Delta)^{n/2}u(y)\,
d\mathcal{H}^n(y)\right|\\
&&\lesssim\int_{\mathbb R^n\setminus
B_{r/2}(o)}\left(r^{1-n}\int_{\partial
B_{r}(o)}\left|\log\frac{|y|}{|x-y|}\right|\,
d\mathcal{H}^{n-1}(x)\right)\big|(-\Delta)^{n/2}u(y)\big|\, d\mathcal{H}^n(y)\\
&&\lesssim \int_{\mathbb R^n\setminus
B_{r/2}(o)}r^{1-n}\left(\int_{\partial
B_{r}(o)\setminus\{x\in{\mathbb R}^n:\
|x-y|\le\theta|y|\}}+\int_{\partial B_{r}(o)\cap\{x\in{\mathbb
R}^n:\
|x-y|\le\theta|y|\}}\right)\\
&&\quad\left|\log\frac{|y|}{|x-y|}\right|\,
d\mathcal{H}^{n-1}\big|(-\Delta)^{n/2}u(y)\big|\, d\mathcal{H}^n(y)\\
&&\lesssim \int_{\mathbb R^n\setminus
B_{r/2}(o)}\big|(-\Delta)^{n/2}u\big|\, d\mathcal{H}^n\to
0\quad\hbox{as}\quad r\to\infty.
\end{eqnarray*}
On the other hand, we have that if $|r\sigma-y|\le|y|/3$, $|y|\ge
r/2$ and $\sigma\in\partial B_1(o)$ then
$$
\left|\log\frac{|y|}{|r\sigma-y|}\right|\le\log\frac{3}{2}+\left|\log\Big|\sigma-\frac{y}{r}
\Big|\right|,
$$
and consequently, if $E_t=\{\sigma\in\partial B_1(o):\
|u_2(r\sigma)|>t\}$ for $t>0$ then

\begin{eqnarray*}
&&t\mathcal{H}^{n-1}(E_t)\\
&&\le\int_{E_t}|u_2|\,d\mathcal{H}^{n-1}\\
&&\lesssim\int_{\mathbb R^n\setminus
B_{r/2}(o)}\left(\int_{E_t}\left|\log\frac{|y|}{|r\sigma-y|}\right|\,d\mathcal{H}^{n-1}(\sigma)\right)\big|(-\Delta)^{n/2}u(y)\big|\, d\mathcal{H}^n(y)\\
&&\lesssim\int_{\mathbb R^n\setminus
B_{r/2}(o)}\left(\int_{E_t\setminus\{\sigma\in\partial B_1(o):\ |r\sigma-y|\le|y|/3\}}+\int_{E_t\cap\{\sigma\in\partial B_1(o):\ |r\sigma-y|\le|y|/3\}}\right)\\
&&\quad\left|\log\frac{|y|}{|r\sigma-y|}\right|\,d\mathcal{H}^{n-1}(\sigma)\big|(-\Delta)^{n/2}u(y)\big|\, d\mathcal{H}^n(y)\\
&&\lesssim\mathcal{H}^{n-1}(E_t)\left(\int_{\mathbb R^n\setminus
B_{r/2}(o)}|(-\Delta)^{n/2}u|\,
d\mathcal{H}^n\right)\Big(1-\log\big(\mathcal{H}^{n-1}(E_t)\big)\Big),
\end{eqnarray*}
and hence
$$
\mathcal{H}^{n-1}(E_t)\lesssim\exp\Big(-\frac{t}{o(1)}\Big)\quad\hbox{as}\quad
r\to\infty.
$$
Note that this last $o(1)$ is positive. So, the layer-cake
representation theorem yields
\begin{eqnarray*}
&&\left|\frac{1}{\mathcal{H}^{n-1}\big(\partial
B_r(o)\big)}\int_{\partial
B_r(o)}\Big(\exp\big(pu_2(x)\big)-1\Big)\,d\mathcal{H}^{n-1}\right|\\
&&=\frac{p}{\mathcal{H}^{n-1}\big(\partial B_r(o)\big)}\int_0^\infty
\mathcal{H}^{n-1}(E_t)\exp(pt)dt=o(1)\quad\hbox{as}\quad r\to\infty.
\end{eqnarray*}
The previously-established equalities and inequalities indicate that
\begin{eqnarray*}
&&\frac{p}{\mathcal{H}^{n-1}\big(\partial B_r(o)\big)}\int_{\partial
B_r(o)}u\,d\mathcal{H}^{n-1}\\
&&=\frac{p}{\mathcal{H}^{n-1}\big(\partial
B_r(o)\big)}\left(\int_{\partial B_r(o)}(u-u_2)\,
d\mathcal{H}^{n-1}\right)+o(1)\\
&&=\log\left(\frac{1}{\mathcal{H}^{n-1}\big(\partial
B_{|x|}(o)\big)}\int_{\partial
B_{|x|}(o)}\exp(pu)\big(\exp(-pu_2)\big)\big)\,
d\mathcal{H}^{n-1}\right)+o(1)\\
&&=\log\left(\frac{1}{\mathcal{H}^{n-1}\big(\partial
B_{|x|}(o)\big)}\int_{\partial B_{|x|}(o)}\exp(pu)\,
d\mathcal{H}^{n-1}\right)+o(1)
\end{eqnarray*}
holds whenever $r\to\infty$, as desired.

(iii) It is clear that $S_{g,n}\ge 0$ is equivalent to $\Delta
u+(n/2-1)|\nabla u|^2\le 0$. Since
$$
\Delta\bar{u}=\frac{1}{\mathcal{H}^{n-1}\big(\partial
B_{|x|}(o)\big)}\int_{\partial B_{|x|}(o)}\Delta
u\,d\mathcal{H}^{n-1}
$$
and (thanks to Cauchy-Schwarz's inequality)
\begin{eqnarray*}
&&|\nabla\bar{u}|^2=\left(\frac{1}{\mathcal{H}^{n-1}\big(\partial
B_{1}(o)\big)}\int_{\partial B_{1}(o)}\frac{d u}{dr}\,d\mathcal{H}^{n-1}\right)^2\\
&&\le\frac{1}{\mathcal{H}^{n-1}\big(\partial
B_{|x|}(o)\big)}\int_{\partial B_{|x|}(o)}|\nabla
u|^2\,d\mathcal{H}^{n-1},
\end{eqnarray*}
one gets that $S_{g,n}\ge 0$ implies
$\Delta\bar{u}+(n/2-1)|\nabla\bar{u}|^2\le 0$ which in turn gives
$S_{\bar{g},n}\ge 0$.

Next, the fact that $Q_{\bar{g},n}$ is absolutely integrable with respect to $dv_{\bar{g}}$ follows from the following estimate (via Fubini's theorem):
\begin{eqnarray*}
&&\int_{\mathbb R^n}|Q_{\bar{g},n}|dv_g=\int_{\mathbb R^n}|(-\Delta)^{n/2}\bar{u}|\,d\mathcal{H}^n\\
&&=\int_{\mathbb
R^n}\left|(-\Delta)^{n/2}\left(\frac{1}{\mathcal{H}^{n-1}\big(\partial
B_{|x|}(o)\big)}\int_{\partial B_{|x|}(o)}u\,d\mathcal{H}^{n-1}\right)\right|\,d\mathcal{H}^n\\
&&=\int_{\mathbb R^n}\left|\frac{1}{\mathcal{H}^{n-1}\big(\partial
B_{|x|}(o)\big)}\int_{\partial B_{|x|}(o)}(-\Delta)^{n/2} u\,d\mathcal{H}^{n-1}\right|\,d\mathcal{H}^n\\
&&\le\frac{1}{\mathcal{H}^{n-1}\big(\partial
B_{1}(o)\big)}\int_{\partial B_{1}(o)}\left(\int_{\mathbb R^n}|(-\Delta)^{n/2} u|\,d\mathcal{H}^n\right)\,d\mathcal{H}^{n-1}\\
&&=\int_{\mathbb R^n}|(-\Delta)^{n/2} u|\,d\mathcal{H}^n=\int_{\mathbb R^n}|Q_{g,n}|dv_g<\infty.
\end{eqnarray*}

Note that (i) and (ii) in Proposition \ref{l2}, together with the completeness of $g=e^{2u}g_0$, yield that $\int_0^\infty e^{u(r\sigma)}dr$ diverges for any given $\sigma\in\partial B_1(o)$ and so that
$$
\int_0^\infty e^{\bar{u}}dr=\frac{1}{\mathcal{H}^{n-1}\big(\partial B_1(o)\big)}\int_{\partial B_{1}(o)}\left(\int_0^\infty e^{u(r\sigma)}dr\right)\,d\mathcal{H}^{n-1}(\sigma)
$$
diverges. Therefore $\bar{g}=e^{2\bar{u}}g_0$ is complete.

(iv) Making a simple calculation with the spherical coordinate system and applying Proposition \ref{l1} (vi) to the conformal metric $\bar{g}$, we immediately obtain
\begin{eqnarray*}
&&\int_{\mathbb R^n}Q_{{g},n}\,dv_{g}=\int_{\mathbb R^n}(-\Delta)^{n/2}{u}\,d\mathcal{H}^{n}=\int_{\mathbb R^n}(-\Delta)^{n/2}\bar{u}\,d\mathcal{H}^{n}\\
&&=\int_{\mathbb R^n}Q_{\bar{g},n}\,dv_{\bar{g}}=\int_{\mathbb R^n}(-\Delta)^{n/2}\bar{u}\,d\mathcal{H}^n\le 2^{n-1}\Gamma(n/2)\pi^{n/2},
\end{eqnarray*}
whence completing the argument.

\end{proof}

\section{Proof of (\ref{eq11a}) -- Special Case}

In this section we verify that (\ref{eq11a}) is true under the
radial symmetry.

\begin{proposition}\label{eq:radial} Let $u\in C^\infty(\mathbb R^n)$ be radially symmetric and satisfy
the hypotheses of Theorem \ref{t1}. If $w(s)=s+u(e^s)$ and

\[
\left\{\begin{array} {r@{\;\quad}l}
V_n(t)=\int_{B_{e^t}(o)}e^{nu}\,d\mathcal{H}^n=n\omega_n\int_{-\infty}^te^{nw(s)}\, ds, &\\
V_{n-1}(t)=\frac1n\int_{\partial B_{e^t}(o)}e^{(n-1)u}\,d\mathcal{H}^{n-1}=\omega_ne^{(n-1)w(t)},&\\
V_{n-2}(t)=\frac{1}{n(n-1)}\int_{\partial
B_{e^t}(o)}H_1e^{(n-1)u}\,d\mathcal{H}^{n-1}=\big(\frac{\omega_n}{n-1}\big)\frac{H_1(e^t)}{e^{(1-n)w(t)}}, &\\
V_{n-3}(t)=\frac{2}{n(n-1)(n-2)}\int_{\partial
B_{e^t}(o)}H_2e^{(n-1)u}\,d\mathcal{H}^{n-1}=\big(\frac{2\omega_n}{(n-1)(n-2)}\big)\frac{H_2(e^t)}{e^{(1-n)w(t)}},
\end{array}
\right.
\]
where $H_k$ stands for the $k$-th symmetric form of the principle
curvature of the boundary of a convex domain in $\mathbb R^n$, then:

\item{\rm(i)}
$$
\lim_{t\to
\infty}\frac{\big(V_{n-3}(t)\big)^{\frac{n-2}{n-1}}}{\omega_n^{\frac{1}{n-1}}\big(V_{n-2}(t)\big)^{\frac{n-3}{n-1}}}
=1-\frac{1}{2^{n-1}\Gamma(n/2)\pi^{n/2}}\int_{\mathbb
R^n}(-\Delta)^{n/2}u\, d\mathcal{H}^n.
$$

\item{\rm(ii)} $\eqref{eq11a}$ holds.
\end{proposition}

\begin{proof} (i) From (ii) and (v) of Proposition\;$\ref{l1}$\; we
read off
$$\lim_{r\to
\infty}r\frac{dv(r)}{dr}=-\frac{1}{2^{n-1}\Gamma(n/2)\pi^{n/2}}\int_{\mathbb
R^n}(-\Delta)^{n/2}u\, d\mathcal{H}^n=\lim_{r\to
\infty}r\frac{du(r)}{dr},
$$
where \;$v$\;is given as in Proposition \ref{l1}.

We now consider cylindrical coordinates $|x|=r=e^t$ and then use
$w(t)=u(e^t)+t$ to get
$$
\frac{d w}{dt}=r\frac{du}{dr}+1\quad\hbox{and}\quad \lim_{t\to
\infty}\frac{d w}{dt}=\lim_{r\to \infty}r\frac{du}{dr}+1,
$$
whence obtaining
$$
\lim_{t\to \infty}\frac{d
w}{dt}=1-\frac{1}{2^{n-1}\Gamma(n/2)\pi^{n/2}}\int_{\mathbb
R^n}(-\Delta)^{n/2}u\, d\mathcal{H}^n.
$$
On the other hand, from the rule of transformation of $H_1$ under
conformal changes and the formula
$n\omega_n=\mathcal{H}^{n-1}\big(\partial B_1(o)\big)$ we see
$$
H_1(r):=H_1[\partial B_r(o)]=
(n-1)e^{-u(r)}\Big(\frac{1}{r}+\frac{du}{dr}\Big).
$$
This equality, plus change of variables and the relation
$$
2H_2(r):=2H_2[\partial B_r(o)]=H_1^2(r)-\hbox{tr}L^2(r),
$$
where $\hbox{tr}L^2(r)$ is the trace of the square of the second
fundamental form of $\partial B_r(o)$ whose $L(r)$ is the
$(n-1)\times(n-1)$ matrix $e^{-u}(r^{-1}+{d u}/{dr})\delta_{ij}$,
easily implies
$$
V_{n-2}(t)=\omega_n e^{(n-2)w(t)}\frac{dw}{dt}\quad\hbox{and}\quad
V_{n-3}(t)=\omega_ne^{(n-3)w(t)}\Big(\frac{dw}{dt}\Big)^2.
$$
Consequently, we find
$$
\frac{\big(V_{n-3}(t)\big)^{\frac{n-2}{n-1}}}{\omega_n^{\frac{1}{n-1}}\big(V_{n-2}(t)\big)^{\frac{n-3}{n-1}}}=\frac{dw}{dt}\\
$$
thereby establishing the required formula:
$$
\lim_{t\to
\infty}\frac{\big(V_{n-3}(t)\big)^{\frac{n-2}{n-1}}}{\omega_n^{\frac{1}{n-1}}\big(V_{n-2}(t)\big)^{\frac{n-3}{n-1}}}=1-\frac{1}{2^{n-1}\Gamma(n/2)\pi^{n/2}}\int_{\mathbb
R^n}(-\Delta)^{n/2}u\, d\mathcal{H}^n.
$$

(ii) Using the definitions of $V_n$ and $V_{n-1}$, we conclude
$$
\frac{\big(V_{n-1}(t)\big)^{\frac{n}{n-1}}}{\omega_n^{\frac{1}{n-1}}V_{n}(t)}=\frac{n^{-1}e^{nw(t)}}{\int_{-\infty}^te^{nw(s)}ds}.
$$
On the other hand, from Proposition\;$\ref{l1}$ (v) with connection
to $\lim_{t\to\infty}dw/dt$ we infer $\lim_{t\to \infty}{d
w}/{dt}\geq 0$. Next we handle two cases:

{\it Case 1: $\lim_{t\to \infty}{d w}/{dt}>0$.} Under this
condition, we clearly have
$$
\lim_{t\to \infty}e^{nw(t)}=\lim_{t\to
\infty}\int_{-\infty}^te^{nw(s)}ds=\infty,
$$
and thus use L'H\^{o}pital's rule to get
$$
\lim_{t\to
\infty}\frac{\big(V_{n-1}(t)\big)^{\frac{n}{n-1}}}{\omega_n^{\frac{1}{n-1}}V_{n}(t)}=\lim_{t\to
\infty}\frac{dw}{dt}.
$$

{\it Case 2: $\lim_{t\to \infty}{d w}/{dt}=0$.} When
\;$\lim_{t\to\infty}V_{n}(t)=\infty$\;, we may have either
$\lim_{t\to\infty}V_{n-1}(t)=\infty$ or
$\sup_{t>0}V_{n-1}(t)<\infty$. For the former we can once again use
L'H\^{o}pital's rule to deduce
$$
\lim_{t\to
\infty}\frac{\big(V_{n-1}(t)\big)^{\frac{n}{n-1}}}{\omega_n^{\frac{1}{n-1}}V_{n}(t)}=\lim_{t\to
\infty}\frac{dw}{dt}=0.
$$
For the latter, we trivially get
$$
\lim_{t\to
\infty}\frac{\big(V_{n-1}(t)\big)^{\frac{n}{n-1}}}{\omega_n^{\frac{1}{n-1}}V_{n}(t)}=0=\lim_{t\to
\infty}\frac{dw}{dt}.
$$
On the other hand, when $\sup_{t>0}V_{n}(t)<\infty$, we have
$\lim_{t\to\infty}e^{nw(t)}=0$ which in turn yields
$$
\lim_{t\to
\infty}\frac{\big(V_{n-1}(t)\big)^{\frac{n}{n-1}}}{\omega_n^{\frac{1}{n-1}}V_{n}(t)}=0=\lim_{t\to
\infty}\frac{dw}{dt}.
$$
All in all, we arrive at
$$
\lim_{t\to
\infty}\frac{\big(V_{n-1}(t)\big)^{\frac{n}{n-1}}}{\omega_n^{\frac{1}{n-1}}V_{n}(t)}=\lim_{t\to
\infty}\frac{dw}{dt}=1-\frac{1}{2^{n-1}\Gamma(n/2)\pi^{n/2}}\int_{\mathbb
R^n}(-\Delta)^{n/2}u\, d\mathcal{H}^n,
$$
as desired.
\end{proof}

\section{Proof of (\ref{eq11a}) -- General Case}

In this section we handle the validity of (\ref{eq11a}) without the
radially symmetric hypothesis.

\begin{proposition}\label{eq:radial2} Let $u\in C^\infty(\mathbb R^n)$ satisfy
the hypotheses of Theorem \ref{t1}. If
\[
\left\{\begin{array} {r@{\;\quad}l}
\mathsf{V}_{n}(r)=\int_{B_r(o)}e^{nu}\,d\mathcal{H}^n, &\\
\mathsf{V}_{n-1}(r)=\frac{1}{n}\int_{\partial B_r(o)}e^{(n-1)u}\,d\mathcal{H}^{n-1}, &\\
\mathsf{V}_{n-2}(r)=\frac{1}{n(n-1)}\int_{\partial
B_r(o)}H_1\frac{d\mathcal{H}^{n-1}}{e^{(1-n)u}}=\frac{1}{n}\int_{\partial
B_r(o)}\left(\frac{1}{r}+\frac{\partial u}{\partial
r}\right)\frac{d\mathcal{H}^{n-1}}{e^{(2-n)u}}, &\\
\mathsf{V}_{n-3}(r)=\frac{2}{n(n-1)(n-2)}\int_{\partial
B_r(o)}H_2\frac{d\mathcal{H}^{n-1}}{e^{(1-n)u}}=\frac{1}{n}\int_{\partial
B_r(o)}\left(\frac{1}{r}+\frac{\partial u}{\partial
r}\right)^2\frac{d\mathcal{H}^{n-1}}{e^{(3-n)u}},
\end{array}
\right.
\]
where $H_k$ still means the $k$-th symmetric form of the principle
curvature of the boundary of a convex domain in $\mathbb R^n$, and
if $\bar{\mathsf{V}}_{n}$, $\bar{\mathsf{V}}_{n-1}$,
$\bar{\mathsf{V}}_{n-2}$ and $\bar{\mathsf{V}}_{n-3}$ denote the
analogously-defined mixed volumes with respect to the conformal
metric $e^{2\bar u}g_0$, where
$$
\bar{u}(x)=\bar{u}(r)=\frac{1}{\mathcal{H}^{n-1}\big(\partial B_r(o)\big)}\int_{\partial B_r(o)}u\,d\mathcal{H}^{n-1},
$$
then:

\item{\rm(i)}
$$
\frac{1}{\mathcal{H}^{n-1}\big(\partial B_r(o)\big)}\int_{\partial B_r(o)}\left(\frac{\partial
u}{\partial
r}\right)^k\,d\mathcal{H}^{n-1}=O\left(\frac{1}{r^k}\right)\quad\text{for}\quad k=1,2,3,4,
$$
and
$$
\frac{1}{\mathcal{H}^{n-1}\big(\partial B_r(o)\big)}\int_{\partial B_r(o)}\left(\frac{\partial
u}{\partial r}\right)^2\,d\mathcal{H}^{n-1}=\left(\frac{\partial \bar u}{\partial
r}\right)^2+o\left(\frac{1}{r^2}\right),
$$
as $r\to\infty$.

\item{\rm(ii)}
$$
\frac{d\mathsf{V}_n(r)}{dr}=\frac{d\bar{\mathsf{V}}_n(r)}{dr}\big(1+o(1)\big)\quad\hbox{and}\quad
\mathsf{V}_{n-1}(r)=\bar{\mathsf{V}}_{n-1}(r)\big(1+o(1)\big)\quad\hbox{as}\quad
r\to\infty.
$$
Moreover,
$$
\mathsf{V}_{n-2}(r)=\bar{\mathsf{V}}_{n-2}(r)\big(1+o(1)\big)\quad\hbox{and}\quad
\mathsf{V}_{n-3}(r)=\bar{\mathsf{V}}_{n-3}(r)\big(1+o(1)\big)\quad\hbox{as}\quad
r\to\infty,
$$
provided
$$
\lim_{r\to\infty}\Big(1+r\frac{\partial\bar{u}}{\partial r}\Big)>0.
$$

\item{\rm(iii)}
$$
\lim_{r\to
\infty}\frac{\big(\mathsf{V}_{n-3}(r)\big)^{\frac{n-2}{n-1}}}{\omega_n^{\frac{1}{n-1}}\big(\mathsf{V}_{n-2}(r)\big)^{\frac{n-3}{n-1}}}
=1-\frac{1}{2^{n-1}\Gamma(n/2)\pi^{n/2}}\int_{\mathbb
R^n}(-\Delta)^{n/2}u\, d\mathcal{H}^n
$$
whenever
$$
1-\frac{1}{2^{n-1}\Gamma(n/2)\pi^{n/2}}\int_{\mathbb
R^n}(-\Delta)^{n/2}u\, d\mathcal{H}^n>0.
$$

\item{\rm(iv)} (\ref{eq11a}) holds.
\end{proposition}

\begin{proof} (i) The argument can be achieved via a slight modification of the proof of \cite[Lemma 3.4]{ChQiYa} -- the details are left for the interested readers.

(ii) The first two relations follow directly from
Proposition\;$\ref{l2}$ (ii). To prove the second two relations, we
will bring the ideas used in proving \cite[Lemma 3.5]{ChQiYa} into
play.

For simplicity, in what follows, let us put $a=\frac{\partial \bar u}{\partial r}$
and $b=e^{\bar u}$. Then from the definition of
$\bar{\mathsf{V}}_{n-2}$ and the easily-checked equation
$\int_{\partial B_r(o)}\left(\frac{\partial u}{\partial
r}-a\right)\,d\mathcal{H}^{n-1}=0$ we get
$$
\bar{\mathsf{V}}_{n-2}(r)=\frac{1}{n}\int_{\partial
B_r(o)}\left(\frac{1}{r}+\frac{\partial u}{\partial
r}\right)e^{(n-2)u}\,d\mathcal{H}^{n-1}=\frac{1}{n}\mathcal{H}^{n-1}\big(\partial
B_r(o)\big)\left(\frac{1}{r}+a\right)b^{n-2},
$$
and consequently,
\begin{eqnarray*}
&&\mathsf{V}_{n-2}(r)-\bar{\mathsf{
V}}_{n-2}(r)\\
&&=\frac{1}{n}\left(\frac{1}{r}+a\right)\int_{\partial
B_r(o)}\left(e^{(n-2)u}-b^{n-2}\right)\,d\mathcal{H}^{n-1}\\
&&+\frac{1}{n}\int_{\partial
B_r(o)}\left(\frac{\partial u}{\partial r}-a\right)e^{(n-2)u}\,d\mathcal{H}^{n-1}\\
&&=\frac{1}{n}\left(\frac{1}{r}+a\right)\int_{\partial
B_r(o)}\left(e^{(n-2)u}-b^{n-2}\right)\,d\mathcal{H}^{n-1}\\
&&+\frac{1}{n}\int_{\partial B_r(o)}\left(\frac{\partial u}{\partial
r}-a\right)\left(e^{(n-2) u}-b^{n-2}\right)\,d\mathcal{H}^{n-1}.
\end{eqnarray*}
Now, using Proposition\;$\ref{l2}$ (ii) and the above-established
formula for $\bar{\mathsf{V}}_{n-2}$ we get
\begin{eqnarray*}
&&\frac{1}{n}\left(\frac{1}{r}+a\right)\int_{\partial
B_r(o)}\left(e^{(n-2)\bar
u}-b^{n-2}\right)\,d\mathcal{H}^{n-1}\\
&&=\frac{1}{n}\left(\frac{1}{r}+a\right)\mathcal{H}^{n-1}\big(\partial
B_r(o)\big)\left(e^{o(1)}-1\right)b^{n-2}\,d\mathcal{H}^{n-1}\\
&&=\bar{\mathsf{V}}_{n-2}(r)o(1).
\end{eqnarray*}
At the same time, a combined application of Cauchy-Schwarz's inequality, the binomial identity, the last-established (i) and Proposition \ref{l2} (ii) yields
\begin{eqnarray*}
&&\int_{\partial B_r(o)}\left(\frac{\partial u}{\partial
r}-a\right)\left(e^{(n-2)u}-b^{n-2}\right)\,d\mathcal{H}^{n-1}\\
&&\leq\left(\int_{\partial B_r(o)}\left(\frac{\partial u}{\partial
r}-a\right)^2\,d\mathcal{H}^{n-1}\right)^{\frac{1}{2}}\left(\int_{\partial
B_r(o)}\left(e^{(n-2)u}-b^{n-2}\right)^2\,d\mathcal{H}^{n-1}\right)^{\frac{1}{2}}\\
&&=\left(\int_{\partial B_r(o)}\left(\left(\frac{\partial
u}{\partial r}\right)^2-2a\frac{\partial u}{\partial
r}+a^2\right)\,d\mathcal{H}^{n-1}\right)^\frac12\\
&&\times\left(\int_{\partial
B_r(o)}\left(e^{(n-2)u}-b^{n-2}\right)^2\,d\mathcal{H}^{n-1}\right)^{\frac{1}{2}}\\
&&=\left(\mathcal{H}^{n-1}\big(\partial B_r(o)\big)o\left(\frac{1}{r^2}\right)\right)^\frac12\left(\int_{\partial
B_r(o)}\left(e^{(n-2)u}-b^{n-2}\right)^2\,d\mathcal{H}^{n-1}\right)^{\frac{1}{2}}\\
&&=\left(\mathcal{H}^{n-1}\big(\partial B_r(o)\big)o\left(\frac{1}{r^2}\right)\right)^\frac12\left(\mathcal{H}^{n-1}\big(\partial B_r(o)\big)b^{2n-4}o(1)\right)^\frac12\\
&&=r^{-1}\mathcal{H}^{n-1}\big(\partial B_r(o)\big)b^{n-2}o(1).
\end{eqnarray*}
As a result, we find
\begin{eqnarray*}
&&\mathsf{V}_{n-2}(r)-\bar{\mathsf{V}}_{n-2}(r)\\
&&=\Big(\bar{\mathsf{V}}_{n-2}(r)+r^{-1}\mathcal{H}^{n-1}\big(\partial
B_r(o)\big)b^{n-2}\Big)o(1)\\
&&=\bar{\mathsf{V}}_{n-2}(r)\Big(1+\frac{n}{ar+1}\Big)o(1)\\
&&=\bar{\mathsf{V}}_{n-2}(r)o(1),
\end{eqnarray*}
thanks to the assumption $\lim_{r\to\infty}(1+ar)>0$. This proves the third relation.

To prove the fourth one, we argue in a similar way. First of all, using the definition of $\bar V_{n-3}$, we get
$$
\bar{\mathsf{V}}_{n-3}(r)=\frac{1}{n}\int_{\partial
B_r(o)}\left(\frac{1}{r}+\frac{\partial u}{\partial
r}\right)^2e^{(n-3)u}\,d\mathcal{H}^{n-1}=\frac{\mathcal{H}^{n-1}\big(\partial
B_r(o)\big)}{nb^{3-n}}\left(\frac{1}{r}+a\right)^2,
$$
and then
\begin{eqnarray*}
&&\mathsf{V}_{n-3}(r)-\bar{\mathsf{V}}_{n-3}(r)\\
&&=\frac{1}{n}\int_{\partial
B_r(o)}\left(\frac{1}{r}+a\right)^2\left(e^{(n-3)u}-b^{n-3}\right)\,d\mathcal{H}^{n-1}\\
&&+\frac{2}{rn}\int_{\partial
B_r(o)}\left(\frac{\partial u}{\partial
r}-a\right)e^{(n-3)u}\,d\mathcal{H}^{n-1}\\
&&+\frac{1}{n}\int_{\partial B_r(o)}\left(\Big(\frac{\partial
u}{\partial r}\Big)^2-a^2\right)e^{(n-3)u}\,d\mathcal{H}^{n-1}.
\end{eqnarray*}
In the sequel, we control the three terms in the last formula. As in the proof of the third relation, using Proposition\;$\ref{l2}$ (ii), we estimate the first term as follows:
$$
\frac{1}{n}\int_{\partial
B_r}\left(\frac{1}{r}+a\right)^2\left(e^{(n-3)u}-b^{n-3}\right)\,d\mathcal{H}^{n-1}=V_{n-3}(r)o(1).
$$
Next, still following the same argument based on Cauchy-Schwarz's
inequality, Proposition\;$\ref{l2}$ (ii) and Proposition \ref{eq:radial2} (i), we
get the estimate for the second term:
$$
\frac{2}{rn}\int_{\partial B_r(o)}\left(\frac{\partial u}{\partial
r}-a\right)e^{(n-3)u}\,d\mathcal{H}^{n-1}=r^{-2}\mathcal{H}^{n-1}\big(\partial
B_r(o)\big)b^{n-3}o(1).
$$
Now, in order to estimate the third term, we firstly employ Young's inequality to get
\begin{eqnarray*}
&&\left|\int_{\partial B_r(o)}\left(\frac{\partial u}{\partial
r}-a\right)\frac{\partial u}{\partial r}e^{(n-3)u}\,d\mathcal{H}^{n-1}\right|\\
&&\leq\left(\int_{\partial B_r(o)}\left(\frac{\partial u}{\partial
r}-a\right)^2\,d\mathcal{H}^{n-1}\right)^{\frac{1}{2}}\\
&&\times\left(\int_{\partial B_r(o)}\left(\frac{\partial u}{\partial
r}\right)^4\,d\mathcal{H}^{n-1}\right)^{\frac{1}{4}}\\
&&\times\left(\int_{\partial
B_r(o)}e^{4(n-3)u}\,d\mathcal{H}^{n-1}\right)^{\frac{1}{4}}.
\end{eqnarray*}
Secondly, we rewrite the third term and use Cauchy-Schwarz's inequality, Proposition \ref{l2} (ii) and Proposition \ref{eq:radial2} (i)
and the finite limit $\lim_{r\to\infty}(1+ra)>0$ to derive
\begin{eqnarray*}
&&\int_{\partial B_r(o)}\left(\Big(\frac{\partial u}{\partial
r}\Big)^2-a^2\right)e^{(n-3)u}\,d\mathcal{H}^{n-1}\\
&&=a\int_{\partial B_r(o)}\left(\frac{\partial u}{\partial
r}-a\right)e^{(n-3)u}\,d\mathcal{H}^{n-1}\\
&&+\int_{\partial B_r(o)}\left(\frac{\partial u}{\partial
r}-a\right)\Big(\frac{\partial
u}{\partial r}\Big)e^{(n-3)u}\,d\mathcal{H}^{n-1}\\
&&\le r^{-2}\mathcal{H}^{n-1}\big(\partial B_r(o)\big)b^{n-3}o(1).
\end{eqnarray*}

With the help of the above estimates and the limit $\lim_{r\to\infty}(1+ar)>0$ we get
$$
\mathsf{V}_{n-3}(r)-\bar{\mathsf{V}}_{n-3}(r)=\mathsf{V}_{n-3}(r)o(1)+r^{-2}\mathcal{H}^{n-1}\big(\partial
B_r(o)\big)o(1)=\mathsf{V}_{n-3}(r)o(1),
$$
completing the proof of the fourth relation.

(iii) Under the given assumption, the proofs of Propositions
\ref{l1} (iv) and \ref{l2} (iv) yield
$$
\lim_{r\to\infty}\Big(1+r\frac{\partial\bar{u}}{\partial r}\Big)
=1-\frac{\int_{\mathbb R^n}(-\Delta)^{n/2}\bar{u}\,
d\mathcal{H}^n}{2^{n-1}\Gamma(n/2)\pi^{n/2}}=1-\frac{\int_{\mathbb R^n}(-\Delta)^{n/2}u\,
d\mathcal{H}^n}{2^{n-1}\Gamma(n/2)\pi^{n/2}}>0.
$$
This fact, along with Proposition\;$\ref{eq:radial}$\;(i) and
Proposition\;$\ref{eq:radial2}$\;(ii), implies
$$
1-\frac{\int_{\mathbb R^n}(-\Delta)^{n/2}u\,
d\mathcal{H}^n}{2^{n-1}\Gamma(n/2)\pi^{n/2}}=\lim_{r\to\infty}\frac{\big(\bar{\mathsf{V}}_{n-3}(r)\big)^{\frac{n-2}{n-1}}}{\omega_n^{\frac{1}{n-1}}\big(\bar{\mathsf{V}}_{n-2}(r)\big)^{\frac{n-3}{n-1}}}=\lim_{r\to\infty}\frac{\big(\mathsf{V}_{n-3}(r)\big)^{\frac{n-2}{n-1}}}{\omega_n^{\frac{1}{n-1}}\big(\mathsf{V}_{n-2}(r)\big)^{\frac{n-3}{n-1}}},
$$
as desired.

(iv) The formula (\ref{eq11a}) is demonstrated through the
equalities
$$
v_g\big(B_r(o)\big)=\mathsf{V}_n(r);\quad
s_g\big(B_r(o)\big)=n\mathsf{V}_{n-1}(r)
$$
and the forthcoming analysis. By Proposition \ref{eq:radial2} (ii),
we have
$$
\lim_{r\to
\infty}\frac{\big(\mathsf{V}_{n-1}(r)\big)^{\frac{n}{n-1}}}{\omega_n^{\frac{1}{n-1}}\mathsf{V}_{n}(r)}=\lim_{r\to
\infty}\frac{\big(\bar
{\mathsf{V}}_{n-1}(r)\big)^{\frac{n}{n-1}}}{\omega_n^{\frac{1}{n-1}}\mathsf{V}_{n}(r)}.
$$

{\it Case 1: $1-\frac{\int_{\mathbb R^n}(-\Delta)^{n/2}u\,
d\mathcal{H}^n}{2^{n-1}\Gamma(n/2)\pi^{n/2}}>0$.} This condition
implies
$$
\lim_{r\to\infty}\bar{\mathsf{
V}}_{n-1}(r)=\lim_{r\to\infty}\mathsf{V}_n(r)=\infty.
$$
Consequently, a combined application of L'H\"{o}pital's rule and
(ii)'s of Propositions \ref{eq:radial2} and \ref{eq:radial} yields
$$
\lim_{r\to
\infty}\frac{\big(\mathsf{V}_{n-1}(r)\big)^{\frac{n}{n-1}}}{\omega_n^{\frac{1}{n-1}}\mathsf{V}_{n}(r)}
=\lim_{r\to \infty}\frac{\frac{d}{dr}\big(\bar
{\mathsf{V}}_{n-1}(r)\big)^{\frac{n}{n-1}}}{\omega_n^{\frac{1}{n-1}}\frac{d}{dr}\bar{\mathsf{V}}_{n}(r)}=1-\frac{\int_{\mathbb
R^n}(-\Delta)^{n/2}u\, d\mathcal{H}^n}{2^{n-1}\Gamma(n/2)\pi^{n/2}}.
$$

{\it Case 2: $1-\frac{\int_{\mathbb R^n}(-\Delta)^{n/2}u\,
d\mathcal{H}^n}{2^{n-1}\Gamma(n/2)\pi^{n/2}}=0$.} With this
hypothesis, we argue as in the radial case, and thus we have to
consider the situation where $\mathsf{V}_n$ is bounded -- under this
boundedness we employ Proposition \ref{eq:radial2} (ii) to derive
$$
\lim_{r\to\infty}\frac{d\mathsf{V}_n(r)}{dr}=0=\lim_{r\to\infty}\frac{d\bar{\mathsf{V}}_n(r)}{dr},
$$
and $\lim_{r\to\infty} r^{n-1}e^{n\bar u}=0$. Hence we obtain
$$
\lim_{r\to
\infty}\frac{\big(\mathsf{V}_{n-1}(r)\big)^{\frac{n}{n-1}}}{\omega_n^{\frac{1}{n-1}}\mathsf{V}_{n}(r)}=0=1-\frac{\int_{\mathbb
R^n}(-\Delta)^{n/2}u\, d\mathcal{H}^n}{2^{n-1}\Gamma(n/2)\pi^{n/2}}.
$$
\end{proof}

\bibliographystyle{amsplain}

\end{document}